\documentclass[12pt]{amsart}
\usepackage{amscd}
\def\frk{\frak}               

\def\Phi{{\frk n}}
\def\Phi{{\frk N}}
\def\opn#1#2{\def#1{\operatorname{#2}}} 
\opn\chara{char} \opn\length{\ell} \opn\pd{pd} \opn\rk{rk}
\opn\projdim{proj\,dim} \opn\injdim{inj\,dim} \opn\rank{rank}
\opn\depth{depth} \opn\sdepth{sdepth} \opn\fdepth{fdepth}
\opn\grade{grade} \opn\height{height} \opn\embdim{emb\,dim}
\opn\codim{codim}  \opn\min{min} \opn\max{max}

\opn\Tr{Tr} \opn\bigrank{big\,rank}
\opn\superheight{superheight}\opn\lcm{lcm}
\opn\trdeg{tr\,deg}
\opn\reg{reg} \opn\lreg{lreg} \opn\ini{in} \opn\lpd{lpd}
\opn\size{size}
\opn\div{div} \opn\Div{Div} \opn\cl{cl} \opn\Cl{Cl}
\opn\Spec{Spec} \opn\Supp{Supp} \opn\supp{supp} \opn\Sing{Sing}
\opn\Ass{Ass} \opn\Min{Min}
\opn\Ann{Ann} \opn\Rad{Rad} \opn\Soc{Soc}
\opn\Im{Im} \opn\Ker{Ker} \opn\Coker{Coker} \opn\Am{Am}
\opn\Hom{Hom} \opn\Tor{Tor} \opn\Ext{Ext} \opn\End{End}
\opn\Aut{Aut} \opn\id{id}  \opn\deg{deg}

\opn\nat{nat}
\opn\pff{pf}
\opn\Pf{Pf} \opn\GL{GL} \opn\SL{SL} \opn\mod{mod} \opn\ord{ord}
\opn\Gin{Gin} \opn\Hilb{Hilb}
\opn\aff{aff} \opn\con{conv} \opn\relint{relint} \opn\st{st}
\opn\lk{lk} \opn\cn{cn} \opn\core{core} \opn\vol{vol}
\opn\link{link} \opn\star{star}
\opn\gr{gr}

\def\pot#1#2{#1[\kern-0.28ex[#2]\kern-0.28ex]}

\opn\dirlim{\underrightarrow{\lim}}
\opn\inivlim{\underleftarrow{\lim}}
\def\Implies{\ifmmode\Longrightarrow \else
        \unskip${}\Longrightarrow{}$\ignorespaces\fi}
\def\implies{\ifmmode\Rightarrow \else
        \unskip${}\Rightarrow{}$\ignorespaces\fi}
\def\iff{\ifmmode\Longleftrightarrow \else
        \unskip${}\Longleftrightarrow{}$\ignorespaces\fi}

\let\:=\colon
\newtheorem{Theorem}{Theorem}[section]
\newtheorem{Lemma}[Theorem]{Lemma}

\let\epsilon\varepsilon
\let\phi=\varphi
\let\kappa=\varkappa
\def\qed{\ifhmode\textqed\fi
      \ifmmode\ifinner\quad\qedsymbol\else\dispqed\fi\fi}
\def\textqed{\unskip\nobreak\penalty50
       \hskip2em\hbox{}\nobreak\hfil\qedsymbol
       \parfillskip=0pt \finalhyphendemerits=0}
\def\dispqed{\rlap{\qquad\qedsymbol}}

\opn\dis{dis}
\def\pnt{{\raise0.5mm\hbox{\large\bf.}}}

\opn\Lex{Lex}

\begin{document}

\title{\bf Stanley Conjecture on intersection of three monomial primary ideals}

\author{Andrei Zarojanu }

\thanks{}

\address{Andrei Zarojanu , Faculty of Mathematics and Computer Sciences,
University of Bucharest, Str. Academiei 14, Bucharest, Romania}
\email{andrei.zarojanu@yahoo.com}

\maketitle
\begin{abstract}  We show that the Stanley's Conjecture holds for an intersection of

 three monomial primary ideals of a polynomial algebra S over a field.

  \vskip 0.4 true cm
 \noindent
  {\it Key words } : Monomial Ideals,  Stanley decompositions, Stanley depth.\\
 {\it 2000 Mathematics Subject Classification: Primary 13C15, Secondary 13F20, 13F55,
13P10.}
\end{abstract}

\section*{Introduction}

Let $K$ be a field and $S=K[x_1,...,x_n]$ be the polynomial ring over $K$ in $n$ variables. Let $I \subset S$ be a monomial ideal of $S, u \in I$ a monomial and $uK[Z], Z \subset \{x_1,...,x_n\}$ the linear $K$-subspace of $I$ of all elements $uf$, $f \in K[Z]$. A presentation of $I$ as a finite direct sum of spaces $\mathcal{D}:I=\bigoplus_{i=1}^{r} u_iK[Z_i]$ is called a Stanley decomposition of $I$. Set $\sdepth (\mathcal{D})= \min \{|Z_i|:i=1,...,r \}$ and
$$\sdepth\ I :=\max\{\sdepth \ ({\mathcal D}):\; {\mathcal D}\; \text{is a
Stanley decomposition of}\;  I \}.$$

 The Stanley's Conjecture \cite{S} says that sdepth $I \geq$ depth $I$. This is proved if either $I$ is an intersection of four monomial prime ideals by \cite[Theorem 2.6]{AP} and \cite[Theorem 4.2]{P1}, or $I$ is the intersection of two monomial irreducible ideals by \cite[Theorem 5.6]{PQ}, or a square free monomial ideal of $K[x_1,\ldots,x_5]$ by \cite{P} (a short exposition on this subject  is given in \cite{P2}). It is the purpose of our paper to show that the Stanley's Conjecture holds for intersections of three monomial primary ideals (see Theorem \ref{3}).
\vskip 1 cm

\section{ Computing depth}

Let $I \subset S$ be a monomial ideal and $I=\bigcap_{i=1}^{s} Q_i$ an irredundant primary decompostion of I, where the $Q_i$ are monomial primary ideals. Set $P_i=\sqrt{Q_i}$.
According to Lyubeznik \cite{L} $\size I$ is the number $v + (n - h) - 1$, where $h$ = height $\sum_{j=1}^{s} Q_j$ and $v$ is the minimum number $t$ such that there exist $1\leq j_1 < ... < j_t\leq s$ with
$$\sqrt{\sum\limits_{k=1}^{t} Q_{j_k}}=\sqrt{\sum\limits_{j=1}^{s} Q_j}.$$ In \cite{L} it shows that $\depth_SI\geq 1+\size I$.

In the study of the Stanley's Conjecture,  we may always assume that $h = n$, that is $\sum_{i=1}^s P_i=m=:(x_1,\ldots,x_n)$, because each free variable on $I$ increases depth and sdepth with $1$.
\begin{Lemma}\label{depth}
 Let $I \subset S$ be a monomial ideal and $I = \bigcap\limits_{i=1}^{3} Q_i$ an irredundant primary decomposition of I, where each $Q_i$ is $P_i$ - primary. Suppose that $P_i\not =m$ for all $i\in [3]$. Then

\begin{enumerate}
  \item [(a)]  If $Q_1 \subset Q_2 + Q_3$ and $P_1\not \subset P_i$ for $i=2,3$,  then \\ $\depth_S S/I = 1 + \min\{\dim S/(P_1+P_2), \dim S/(P_1+P_3)\}.$
  \item [(b)] If $Q_1 \subset Q_2 + Q_3$  and  $P_1\subset P_2$,$P_1\not \subset P_3$, then \\ $\depth_S S/I =  \min\{\dim S/P_2, 1+\dim S/(P_1+P_3)\}.$
  \item [(c)] If $Q_1 \subset Q_2 + Q_3$ and $P_1\subset P_i$ for $i=2,3$ then \\ $\depth_S S/I =  \min\{\dim S/P_2, \dim S/P_3\}.$
  \item [(d)]  If $Q_i \not\subset \sum\limits_{j=1, \ j \neq i}^{3} \ Q_j,$ for all $i$ then  $\depth_S S/I = 1$ if and only if $\size I = 1$.
  \item [(e)]  If $Q_i \not\subset \sum\limits_{j=1, \ j \neq i}^{3} \ Q_j,$ for all $i$ then  $\depth_S S/I = 2$ if and only if $\size I = 2$.
 \end{enumerate}
 \end{Lemma}
\begin{proof}
As  $\Ass_S S/I = \{ P_1, P_2, P_3\}$ we get  $\depth_S S/I>0$ by assumptions.
We have the following  exact sequences
  \begin{enumerate}
   \item{} $$ 0 \rightarrow \frac{S}{I} \rightarrow \frac{S}{Q_1 \cap Q_2} \oplus \frac{S}{Q_1 \cap Q_3} \rightarrow \frac{S}{Q_1} \rightarrow 0,$$
 \item{} $$ 0 \rightarrow \frac{S}{Q_1 \cap Q_2} \rightarrow \frac{S}{Q_1} \oplus \frac{S}{Q_2} \rightarrow \frac{S}{Q_1+Q_2} \rightarrow 0,$$
 \item{} $$  0 \rightarrow \frac{S}{Q_1 \cap Q_3} \rightarrow \frac{S}{Q_1} \oplus \frac{S}{Q_3} \rightarrow \frac{S}{Q_1+Q_3} \rightarrow 0.$$
 \end{enumerate}
Apply Depth Lemma in (2) and (3).
If $P_1$ is not properly contained in $P_2$ or $P_3$ then $\depth \frac{S}{Q_1 \cap Q_3} = 1 + \depth \frac{S}{Q_1+Q_3}$ and  $\depth \frac{S}{Q_1 \cap Q_2} = 1 + \depth_S \frac{S}{Q_1+Q_2}$. If $P_1 \subset P_2$ then $\depth_S \frac{S}{Q_1 \cap Q_2} \geq  \depth_S \frac{S}{Q_2}
 = \dim \frac{S}{P_2}$. But $\depth_S \frac{S}{Q_1 \cap Q_2} \leq  \dim \frac{S}{Q_2}$, that is $\depth_S \frac{S}{Q_1 \cap Q_2} = \dim \frac{S}{P_2}$. Similarly, $\depth_S \frac{S}{Q_1 \cap Q_3} =  \dim \frac{S}{P_3}$ if $P_1\subset P_3$.

The statements (a),(b), (c) follow if we show that
 $$\depth_S S/I = \min \{  \depth_S  \frac{S}{Q_1 \cap Q_2},  \depth_S \frac{S}{Q_1 \cap Q_3} \}.$$
 If
 $\depth_S \frac{S}{Q_1} > \min \{  \depth_S \frac{S}{Q_1 \cap Q_2},  \depth_S \frac{S}{Q_1 \cap Q_3} \}$ then by Depth Lemma applied in (1) we get
 the above equality.
 If $\depth_S \frac{S}{Q_1} = \min \{  \depth_S  \frac{S}{Q_1 \cap Q_2},  \depth_S \frac{S}{Q_1 \cap Q_3} \}$ then we get similarly $\depth_S S/I\geq \depth_S S/Q_1 = \depth_S S/P_1$. As $P_1 \in  \Ass S/I$ then $\depth_S S/I \leq \dim S/P_1 =  \depth_S S/Q_1 .$ Thus $\depth_S S/I = \depth_S \frac{S}{Q_1}$, which is enough.

 (d) If $\depth_S S/I = 1$ then $2 = \depth_S I \geq 1 + \size I$, that is $1 \geq \size I  \geq 0$. But $\size I \not = 0$ because the primary decomposition is irredundant. Conversely, if $\size I = 1$ then $v = 2$ and we may assume that $ P_2 + P_3 = P_1 + P_2 + P_3=m$.
We consider the exact sequences
\begin{enumerate}
\item[(4)]  $$ 0 \rightarrow \frac{S}{I} \rightarrow \frac{S}{Q_1 \cap Q_2} \oplus \frac{S}{Q_3} \rightarrow \frac{S}{Q_3+(Q_1 \cap Q_2)} \rightarrow 0,$$
 \item[(5)] $$ 0 \rightarrow \frac{S}{Q_1 \cap Q_2} \rightarrow \frac{S}{Q_1} \oplus \frac{S}{Q_2} \rightarrow \frac{S}{Q_1+Q_2} \rightarrow 0.$$
\end{enumerate}
 From  (5) we have $\depth_S \frac{S}{Q_1 \cap Q_2} = 1 + \depth_S \frac{S}{Q_1+Q_2} \geq 1$ by Depth Lemma. Note that $\depth_S S/Q_3 \geq 1$ and $\depth_S\frac{S}{Q_3+(Q_1 \cap Q_2)} = \depth_S \frac{S}{(Q_1+Q_3) \cap (Q_2+Q_3)}=0$ because $\sqrt{Q_2+Q_3}=m$, and $Q_1 \not\subset
  Q_2+Q_3$. Thus  Depth Lemma applied in (4) gives $\depth_S S/I = 1$.

 (e) If $\depth_S S/I = 2$, then $\depth_S I = 3 \geq 1 + \size I$. But $\size I \leq 1$ was the subject of (d), so $\size I = 2$. Conversely,  suppose that $\size I = 2$, that is $v = 3$. Then $P_i \not\subset \sum\limits_{j=1, \ j \neq i}^{3} \ P_j,$ for all $i$ and by \cite[Proposition 2.1]{I} we get $\depth_S I\leq 3$. As
 $\depth_S I  \geq 1 + \size I$ we get $\depth_S S/I=2$.
\end{proof}
\vskip 1 cm
\section{ Stanley's depth}

In this section we introduce a new way of splitting, inspired from [4], that helps us to prove the Stanley Conjecture when  $I = \bigcap\limits_{i=1}^{3} Q_i$ is an irredundant primary decomposition of I.

\begin{Theorem} \label{2} Let $I$ be a monomial ideal and $I=Q_1 \cap Q_2$ an irredundant primary decomposition of $I$ , where $Q_i$ is $P_i$ primary. Then the Stanley conjecture holds for $I$.
\end{Theorem}
\begin{proof}
As usual we my suppose that $P_1 + P_2 = m$. Also we may suppose that $P_i\not =m$ for all $i$, because otherwise $\depth_S I=1$ and there exists nothing to show. Applying Depth Lemma in the above exact sequence (2) we get $\depth_S S/I = 1$, so $\depth_S I = 2 = 1 + \size I$. By \cite[Theorem 3.1]{HPV} we have $\sdepth_SI \geq \depth_S I$.
\end{proof}
\begin{Theorem} \label{3} Let $I$ be a monomial ideal and  $I = \bigcap\limits_{i=1}^{3} Q_i$ an irredundant primary decomposition of $I$ , where $Q_i$ is $P_i$ primary. Then the Stanley conjecture holds for $I$.
\end{Theorem}
\begin{proof} We may suppose as above  $P_1+P_2+P_3=m$ and $P_i\not =m$ for all $i$.
If $Q_i \not\subset \sum\limits_{j=1, \ j \neq i}^{3} \ Q_j$, for all $i \in [3]$ we have according to Lemma \ref{depth} minimal depth that is $\depth I = 1 + \size I$. Then by \cite[Theorem 3.1]{HPV} we get $\sdepth_S I\geq \depth_S I$.
Now suppose that  $Q_1 \subset Q_2 + Q_3$. It follows that  $\size I$ = 1. If $P_1 + P_2 = m$ or $P_1 + P_3 = m$ then $\dim \frac{S}{Q_1+Q_2} = 0$ or  $\dim\frac{S}{Q_1+Q_3} = 0$ therefore $\depth_SS/I = 1$ that is $\depth_S I = 2$. Then again we get $\sdepth_S I \geq 1 + \size I = 2 = \depth_S I$ by by \cite[Theorem 3.1]{HPV}.

Otherwise $P_1 + P_2 \neq m\not = P_1 + P_3$. Let $P_1 = (x_1,...,x_r)$ and $P_3 = (x_{e+1},...,x_t)$, $2 \leq r \leq n - 1,e+1 \leq r$. If $r = 1$ then $Q_1 \subset Q_2$ or $Q_1 \subset Q_3$ because $Q_1 \subset Q_2 + Q_3$. This is  false since the primary decomposition is irredundant. If $r = n$ then $P_1 = m $, which is not possible. If $e+1>r$ then $Q_1 \subset Q_2$, also a contradiction.
 We will prove this case by induction on $n$. If $n = 3$, then $\sdepth_S I \geq  1 + \size I = 2 \geq \depth_S I$, because $I$ is not principal. Assume now $n>3$. We set $S'=K[x_1,...,x_r]$, ${\bar S}:=K[x_1,...,x_e,x_{r+1},...,x_n]$
 	and  $J_3 = \bigoplus\limits_{w} w ((I:w) \cap {\bar S})$, where $w$ runs in the finite set of monomials of $K[x_{e+1},...,x_r] \setminus Q_3$.

We claim that $I=Q_1 \cap Q_2 \cap (Q_3 \cap S')S \oplus J_3$.
It is enough to see the inclusion
$"\subset"$.  Let $a \in I$ be a monomial, then $a=uv$, where $u \in {\bar S}$ and $v \in K[x_{e+1},...,x_r]$ are monomials.
If $v \not\in Q_3$  then $u \in (I: v) \cap {\bar S}$, so $a \in J_3$.
If $v \in Q_3$ then   $a \in (Q_3 \cap S')S$. As $a \in I$ we get  $a \in Q_1\cap Q_2$ therefore a $\in Q_1 \cap Q_2 \cap (Q_3 \cap S')S$. The above sum is direct. Indeed,
let $a=uv \in Q_1 \cap Q_2 \cap (Q_3 \cap S')S \cap J_3$  be as above.  Then $v\not \in Q_3$ because $a \in J_3$. But $v$ must be in $ (Q_3 \cap S')S$. Contradiction!

The ideal $I':=Q_1 \cap Q_2 \cap (Q_3 \cap S')S \subset P_1+P_2\not =m$ and so is an extension of an ideal from less than $n$-variables and we may apply the induction hypothesis for $I'$, that is $\sdepth_SI'\geq \depth_SI'$. Since $\sdepth_SI\geq \min\{\sdepth_SI',\{\sdepth_{\bar S} ((I:w) \cap {\bar S})\}_w\}$ it remains to show that $\depth_SI'\geq \depth_SI$ and $\depth_{\bar S} ((I:w) \cap {\bar S})\geq \depth_SI$, applying again the induction hypothesis since $\bar S$ has less than $n$-variables. The first inequality follows because  $\dim S/(P_3\cap S')S\geq \dim S/P_3$, $\dim S/(P_1+(P_3\cap S')S)\geq \dim S/P_1+P_3$ using  Lemma \ref{depth} (a), (b), (c).

For the second inequality note that for $w\not \in Q_1\cup Q_2\cup Q_3$ we have $(Q_i:w)$ primary and so $L_i:=(Q_i:w)\cap {\bar S}$ is ${\bar P}_i:=P_i\cap {\bar S}$-primary too. We have $\dim {\bar S}/{\bar P_i}= \dim S/P_i$ for $i=1,3$ because $(x_{e+1},\ldots,x_r)\subset P_1\cap P_3$. Thus $\dim {\bar S}/({\bar P_1}+{\bar P_i})=\dim S/(P_1+P_i)$ for all $i=2,3$. Using Lemma \ref{depth} we are done because $\dim S/P_2$ appears in the formulas only when $P_1\subset P_2$, that is when  $\dim {\bar S}/{\bar P_2}= \dim S/P_2$.

If $w\in Q_2\setminus(Q_1\cup Q_3)$ then
$$\depth_{\bar S}{\bar S}/(L_1\cap L_3)=1+\dim {\bar S}/({\bar P}_1+ {\bar P}_3)=1+\dim  S/( P_1+  P_3)\geq \depth_S S/I $$
by the same lemma, the only problem could appear when $P_1\subset P_3$, but in this case
 $$\dim  {\bar S}/({\bar P}_1+ {\bar P}_3)=\dim S/(P_1+P_3)={\bar S}/{\bar P}_3=\dim S/P_3$$
 and it follows
$$\depth_{\bar S}{\bar S}/(L_1\cap L_3)=1+\dim {\bar S}/({\bar P}_1+ {\bar P}_3)>\dim S/P_3\geq \depth_S S/I.$$
 If $w\in (Q_1\cap Q_2)\setminus Q_3$ then $\depth_{\bar S} {\bar S}/L_3=\dim S/P_3\geq \depth_S S/I$  by \cite{BH}.
\end{proof}


\vskip 0.5 cm

\begin{thebibliography}{99}

\bibitem{BH} W.\ Bruns and J. Herzog, {\em Cohen-Macaulay rings} Revised edition. Cambridge University Press (1998).
\bibitem{HVZ} J.\ Herzog, M.\ Vladoiu, X.\ Zheng, {\em How to compute the Stanley depth of a monomial ideal,}  J.  Algebra, 322 (2009), 3151-3169.
\bibitem{HPV} J. Herzog, D.  Popescu, M. Vladoiu, {\em Stanley depth and size  of a monomial ideal}, arXiv:AC/1011.6462v1,  2010, to appear in Proceed. AMS.
\bibitem{I} M. Ishaq, {\em Values and bounds of the Stanley depth}, to appear in Carpathian J. Math., arXiv:AC/1010.4692.
\bibitem{L} G.\ Lyubeznik, {\em On the Arithmetical Rank of Monomial ideals}, J.~Algebra {\bf 112}, 86-89  (1988).

\bibitem{AP} A.\ Popescu, {\em Special Stanley Decompositions}, Bull. Math. Soc. Sc. Math. Roumanie, 53(101), no 4 (2010), arXiv:AC/1008.3680.
\bibitem{P} D.\ Popescu,  {\em An inequality between depth and Stanley depth}, Bull. Math. Soc. Sc. Math. Roumanie 52(100), (2009), 377-382, arXiv:AC/0905.4597v2.
\bibitem{P1} D.\ Popescu, {\em Stanley conjecture on intersections of four monomial prime ideals}, arXiv.AC/1009.5646.
\bibitem{P2} D.\ Popescu, {\em Bounds of Stanley depth}, An. St.  Univ. Ovidius. Constanta, 19(2),(2011), 187-194.
\bibitem{PQ} D. Popescu, I. Qureshi, {\em Computing the Stanley depth}, J. Algebra, 323 (2010), 2943-2959.
\bibitem{S} R.\ P.\ Stanley, {\em Linear Diophantine equations and local cohomology,} Invent. Math. 68 (1982) 175-193.
\end{thebibliography}
\end{document}